\documentclass[12 pt,reqno]{amsart}
\usepackage{latexsym, amsmath, mathrsfs, bm, amssymb, longtable, booktabs,amscd,microtype,booktabs,cases,amsfonts,mathrsfs}
\usepackage[square,numbers]{natbib}
\usepackage{graphicx}
\usepackage[dvipsnames]{xcolor}
\usepackage{caption}
\usepackage{dsfont}
\usepackage[all]{xy}
\usepackage{enumerate}
\usepackage{amsthm}
\usepackage{multirow}
\usepackage{tikz}
\usetikzlibrary{matrix,arrows,decorations.pathmorphing}
\usepackage{collectbox}
\usepackage{enumerate}
\usepackage[colorinlistoftodos,prependcaption, textsize=tiny]{todonotes}
\reversemarginpar
\usepackage{amsmath}
\usepackage{enumitem}
\usepackage{hyperref}
\let\oldtocsection=\tocsection

\let\oldtocsubsection=\tocsubsection

\let\oldtocsubsubsection=\tocsubsubsection

\renewcommand{\tocsection}[2]{\hspace{0em}{\vspace{0.5em}}\oldtocsection{#1}{#2}}
\renewcommand{\tocsubsection}[2]{\hspace{1em}{\vspace{0.5em}}\oldtocsubsection{#1}{#2}}
\renewcommand{\tocsubsubsection}[2]{\hspace{2em}\oldtocsubsubsection{#1}{#2}}
\hypersetup{
	colorlinks,
	citecolor=red,
	filecolor=black,
	linkcolor=blue,
	urlcolor=black
}

\numberwithin{equation}{section} 
\makeatother

\newtheorem{thm}{Theorem}[section]
\newtheorem{theorem}[thm]{Theorem}

\newtheorem{prop}[thm]{Proposition}

\theoremstyle{definition}
\newtheorem{definition}[thm]{Definition}
\newtheorem{remark}[thm]{Remark}

\theoremstyle{definition}

\theoremstyle{remark}

\theoremstyle{remark}

\makeatletter
\def\imod#1{\allowbreak\mkern10mu({\operator@font mod}\,\,#1)}
\makeatother

\newcommand{\rhoc}{\rhob}

\newcommand{\omd}[1][d]{\omega^{\,#1}}
\newcommand{\hwvec}{v_{\Lambda_0}}
\newcommand{\wakim}[1][\lambda]{\mathbb{W}_{k}(#1)}
\newcommand{\fock}[1][1]{\pi_{#1}}
\newcommand{\hds}{H^0_{\scriptscriptstyle DS}}
\newcommand{\vacm}{L(\Lambda_0)}
\newcommand{\wverma}[1][\bar{\lambda}]{\mathbf{M}({\gamma_{#1}})}
\newcommand{\lverma}[1][\bar{\lambda}]{\mathbf{L}({\gamma_{#1}})}
\newcommand{\complex}{\mathbb{C}}
\newcommand{{\mg}}{\mathfrak{g}}
\newcommand{{\fin}}{\bar{\mathfrak{g}}}
\newcommand{{\aff}}{\mathfrak{g}}
\newcommand{\prinheis}{\mathfrak{s}}
\newcommand{\mhk}{\mathfrak{h}_k}
\newcommand{\mh}{\mathfrak{h}}
\newcommand{\mhb}{\bar{\mathfrak{h}}}
\newcommand{\wk}[1][k]{\mathscr{W}_{#1}(\bar{\mg})}
\newcommand{\walg}{\mathscr{W}}

\newcommand{\hc}{h^{\vee}}
\newcommand{\rhob}{\bar{\rho}}
\newcommand{\lb}{\bar{\lambda}}
\newcommand{\coxnum}{h^\vee} 
\newcommand{\lrb}{{\lambda} - (k+\hc) \rhob}
\newcommand{\lbrb}{\bar{\lambda} - (k+\hc) \rhob}

\newcommand{\integers}{\mathbb{Z}}

\DeclareMathOperator{\ad}{ad}
\DeclareMathOperator{\rank}{rank}
\DeclareMathOperator{\ch}{char}
\DeclareMathOperator{\tr}{tr}
\DeclareMathOperator{\gr}{gr}
\DeclareMathOperator{\Hilb}{\mathbf{H}}

\begin{document}
\title{The affine Brylinski filtration and $\mathscr{W}$-algebras}
\author{Suresh Govindarajan}
\address{Department of Physics, Indian Institute of Technology Madras, Chennai, India}
\email{suresh@physics.iitm.ac.in}

\author{Sachin S. Sharma}
\address{Department of Mathematics and Statistics, Indian Institute of Technology, Kanpur, India}
\email{sachinsh@iitk.ac.in}

\author{Sankaran Viswanath}
\address{The Institute of Mathematical Sciences, Chennai, India and Homi Bhabha National Institute, Mumbai, India}
\email{svis@imsc.res.in}

\keywords{Vertex algebras, affine Lie algebras, $\mathscr{W}$-algebras}
\maketitle

\noindent
\begin{abstract}
The Brylinski-Kostant filtration on a representation of a finite-dimensional semisimple Lie algebra has interpretations in terms of the algebra, geometry and combinatorics of 
the representation. Its extension to affine Lie algebras was first studied by Slofstra. Recent work of the present authors constructed a Poincar\'{e}-Birkhoff-Witt type basis for the dominant weight spaces of the basic representation of affine Lie algebras of type $A$, which is compatible with the affine Brylinski filtration. In this paper, we overcome the constraint of type dependence, and furnish a new, uniform proof which holds for all simply-laced affine Lie algebras.  

\end{abstract}

\bigskip
\noindent

\section{Introduction}
Let $\fin$ be a simply-laced finite-dimensional simple Lie algebra (i.e., of type $A, D, E$) with Cartan subalgebra $\mhb$.
%
Let $L(\lambda)$ denote an irreducible highest weight representation of $\fin$. The principal nilpotent element of $\fin$ induces a filtration on the weight spaces of $L(\lambda)$, called the {\em Brylinski-Kostant} filtration \cite{RB}.

When $\mu$ is a dominant weight of $\fin$, the Poincar\'{e} series of the Brylinski-Kostant filtration  on the $\mu$-weight space of $L(\lambda)$ coincides with Lusztig's $t$-analog of weight multiplicity $m^{\lambda}_{\mu}(t)$ \cite{RB, broer}.
Braverman and Finkelberg \cite{BF10} conjectured an extension of this result for the untwisted affine Kac-Moody algebras. Slofstra \cite{Slofstra} showed that their conjecture was false, but that it holds if one replaces the principal nilpotent element by the positive part of the principal Heisenberg subalgebra of the affine Lie algebra.  The resulting filtration, termed the (affine) Brylinski filtration was studied in \cite{GSV} for the basic representation $L(\Lambda_0)$ of the affine Kac-Moody algebras of type $A$. 

One of the main results of \cite{GSV} 
was an explicit description of the subspaces of the Brylinski filtration on the dominant weight spaces of $L(\Lambda_0)$ for type $A$. 
%
%
 The proof in \cite{GSV} was limited to type $A$ since it relied on a result of \cite{FKRW} relating the representations of $\mathscr{W}(\mathfrak{gl}_{\ell +1})$ to the representations of $\mathscr{W}_{1 + \infty}$, the universal central extension of the Lie algebra of regular differential operators on $\mathbb{C}^{\times}$. 

In this article, we extend this result to all simply-laced affine Lie algebras (i.e., of types $A_\ell^{(1)}$, $D_\ell^{(1)}$ or $E_n^{(1)}$, $n=6,7,8$) and give a new, type-independent proof that provides an alternate approach even in type A. 
As in \cite{GSV}, the crux of our arguments lies in establishing that the subspace  $L(\Lambda_0)^{\mhb}$ of $\mhb$-invariants is isomorphic to an irreducible Verma module of $\mathscr{W}(\fin)$. We use results of  Arakawa \cite{arakawa2007} which relate representations of $\walg(\fin)$ to those of $\aff$ via the Drinfeld-Sokolov reduction functor; in particular this gives us a criterion for irreducibility of Verma modules of $\walg(\fin)$ (Proposition \ref{prop:wver-irr}). Another key ingredient is the {\em Miura map}, which provides a free field realization of $\walg(\fin)$ in non-critical level \cite{acl}. Pull-backs under the Miura map of oscillator representations of the Heisenberg algebra generically give irreducible Verma modules of $\walg(\fin)$ (Theorem~\ref{thm:lrbgenheis}). The  PBW type basis of the resulting Verma module gives us the desired bases for the subspaces of the Brylinski filtration and establishes our main result. 

The paper is organized as follows. Section~\ref{sec:notn} recalls relevant background from \cite{GSV} and states our main results (Theorems~\ref{thm:mainthm}-\ref{thm:wfock}). Section~\ref{sec:wstructure} contains properties of $\walg$-algebras of relevance to our work, specifically the Miura map. Section~\ref{sec:wreps} specializes work of Arakawa \cite{arakawa2007} to obtain conditions under which Verma modules of $\walg$ are irreducible. Finally, Section~\ref{sec:heismodw} studies pullbacks under the Miura maps of Heisenberg algebra representations and completes the proof of our theorems.

{\bf Acknowledgments:} The authors would like to thank Naoki Genra for informing them about the reference \cite{KW22} pertaining to images of conformal vectors under the Miura map. The authors would also like to thank Bojko Bakalov for his prompt and insightful answers to their questions.
The second named author would like to thank IMSc Chennai for its hospitality while part of this work was being done.

\section{Notation and preliminaries}\label{sec:notn}
\subsection{} Let $\aff$ be an untwisted affine Lie algebra of type $A, D$ or $E$, with Cartan subalgebra $\mh$ and  triangular decomposition $\aff = \mathfrak{n}^- \oplus \mh \oplus \mathfrak{n}^+$.  Let the standard Chevalley generators of $\aff$ be $e_i, f_i \; (0 \leq i \leq \ell)$, and let $K, d$ denote the canonical central element and derivation respectively. The principal Heisenberg subalgebra $\prinheis$ of $\aff$ is defined as:
\[
\prinheis = \{x \in \aff: [x,e] \in \complex K\},
\]
where $e=\sum_{i=0}^\ell e_i$ is the principal nilpotent element. We set $\prinheis^{\pm}:=\prinheis \cap \mathfrak{n}^{\pm}$.

\subsection{} 
Let $\fin$ be the underlying finite dimensional simple Lie algebra of $\aff$ and let $\mhb$ be its Cartan subalgebra. Let $\vacm$ denote the level 1 vacuum module ({\em the basic representation}) of $\aff$ and let $Z = \vacm^{\mhb}$ be the space of $\mhb$-invariants. It is the direct sum of the dominant weight spaces of $\vacm$:
\[ Z= \bigoplus_{n \geq 0} Z_n := \bigoplus_{n \geq 0} L(\Lambda_0)_{\Lambda_0 - n \delta}. \]
where $\delta$ is the null root of $\aff$.


We now recall the (affine) Brylinski filtration \cite{Slofstra} on $\vacm$. 
The subspaces $F^i \vacm$ of the Brylinski filtration are defined (for $i \geq -1$) by:
\begin{equation*}
F^i \vacm = \{ v \in \vacm: x^{i+1} v=0 \text{ for all } x \in \prinheis^+\}.
\end{equation*}

\subsection{} We restrict the Brylinski filtration to $Z$ and let $F^i Z_n := Z_n \cap F^i \vacm$.
The associated graded space $\gr Z$ is then a bi-graded vector space with Hilbert-Poincar\'{e} series \cite{GSV}:
\begin{equation} \label{eq:hilbZ}
  \Hilb(\gr Z; \,t,q) := \sum_{n \geq 0} \sum_{i \geq 0} \dim \left( F^i Z_n \,/ F^{i-1}  Z_n\right) \, t^i q^n = \prod_{k=1}^{\ell} \prod_{n=1}^{\infty}{(1-t^{d_{k}}q^n)^{-1}}.
\end{equation}
where $d_{k}\,(1\leq k\leq \ell)$ are the degrees of the underlying finite-dimensional
simple Lie algebra $\fin$ arranged in ascending order \cite[Chapter 3, Section 3]{Humphreys}.
Our main result (Theorem~\ref{thm:mainthm}) constructs a natural basis of $F^i Z_n$, with $q$-degrees and $t$-degrees that are ``compatible" with Equation~\eqref{eq:hilbZ}.

\subsection{}  Consider the Heisenberg Lie algebra $\widehat{\mh} =
\bar{\mh} \otimes \complex[t, t^{-1}] \oplus \complex K$ and its Fock
space
$$\pi_1 = \mathrm{Ind}_{{\widehat{\mh}}^{+} \oplus \complex
K}^{\widehat{\mh}}{\complex} = U(\widehat{\mh})
\otimes_{U({\widehat{\mh}}^{+} \oplus \complex K)} \complex.$$
Let $Q$ be the root lattice of $\fin$ and $\varepsilon: Q \times Q \rightarrow \{\pm 1\}$ be a bimultiplicative
cocycle with the property $\varepsilon(\alpha, \beta) \varepsilon(\beta,
\alpha) = (-1)^{( \beta \mid \alpha)}.$ Let $\complex_{\varepsilon} [Q]$
be the twisted group algebra
with multiplication $e^{\alpha}e^{\beta} = \varepsilon(\alpha, \beta)
e^{\alpha + \beta}$. The lattice vertex algebra $V_{Q}$ defined as
$$V_{Q} = \pi_1 \otimes_{\complex}{\complex_{\varepsilon} [Q]}.$$
The action of $\widehat{\mh}$ on $1 \otimes e^{\beta}$ is given by $(h
\otimes t^n) (1 \otimes e^{\beta}) = \delta_{n0} \, \beta(h)\,
(1 \otimes e^{\beta})$ for $n \geq 0$ and $h \in \bar{\mh}$.
We write $ht^n$ for $h \otimes t^n$.  For $v \in V_Q$, let the state-field correspondence be denoted $Y(v, z) = \sum_{n \in \integers} v_{(n)} \,z^{-n-1}$. We have 
$$Y( h t^{-1} \otimes 1, z) = \sum_{s \in \mathbb{Z}} h_s z^{-s-1},$$
where $h_s$ denotes the operator $h t^s \otimes \mathrm{id}$ on $V_{Q}$. Also
$$Y (1 \otimes e^{\beta}, z) = e^{\beta} z^{\alpha_0} \,\mathrm{exp}\left(
-\sum_{n <0}{\frac{z^{-n}}{n} \beta_n} \right)\,\mathrm{exp}\left( -\sum_{n
>0}{\frac{z^{-n}}{n} \beta_n} \right), $$
$h \in \bar{\mh}$ and $\beta \in Q$. In the above expression, $e^{\beta}$ is
the left multiplication operator by $1 \otimes e^\beta$ and $z^{\beta_0}$
acts as $z^{\beta_0}(\eta \otimes e^{\alpha}) = z^{(\beta | \alpha)} (\eta
\otimes e^{\beta}),$
where $\eta \in \pi_1$. We identify $\mhb$ and $\mhb^*$ via the normalized Killing form.

\subsection{} The $\walg$-algebra is the vertex subalgebra of $V_Q$ defined by:
\begin{equation}\label{eq:wdef}
\walg := \bigcap_{X \in \fin} \ker_{V_Q} X_{(0)} = \pi_1 \cap \bigcap_{i=1}^{\rank \fin} \ker_{V_Q} (f_i)_{(0)}.
\end{equation}
Here $\pi_1 = \bigcap_{h \in \bar{\mh}} \ker_{V_Q} h_{(0)}$ is the Fock space of the homogeneous Heisenberg subalgebra of $\aff$ and is isomorphic to a Heisenberg vertex algebra.  In \eqref{eq:wdef}, we identify $\fin$ with the subspace of $V_Q$ spanned by $\mhb$ and the $1 \otimes e^\alpha$ for $\alpha$ a root of $\fin$. 

Let $d_1 \leqslant d_2 \leqslant \cdots \leqslant d_\ell$ be the list of degrees of $\fin$. We recall the following important theorem, due to Feigin and Frenkel \cite{feigin-frenkel}:
\begin{theorem}[Feigin-Frenkel]\label{thm:ff} 
There exist elements $\omd[i] \in \walg$ of degree $d_i$ such that $\walg$ is freely generated by  $\omd[1],\, \omd[2],\, \cdots,\, \omd[\ell]$ as a vertex algebra.
\end{theorem}
The field $Y(\omd[i],z)$ has conformal weight $d_i$ and we write:
\[ Y(\omd[i], z) = \sum_{n \in \integers} \omd[i]_n \,z^{-n - d_i}.\]

\subsection{} The module $\vacm$ admits many realizations, one for each inner automorphism of $\fin$ \cite{KKLW}. Let $\sigma$ denote the Coxeter element in the Weyl group $W(\fin)$. This extends to inner automorphisms of $\fin$ and of the lattice vertex algebra $V_Q$; we denote both of these by $\sigma$ again.

Likewise, consider the inner automorphism $\zeta:=\exp\ad \left(2\pi i \rhoc^{\vee}/\hc\right)$  of $\fin$, where $\rhoc^{\vee}$ is a half-sum of positive coroots of $\fin$ and $\hc$ is the dual Coxeter number of $\fin$; this extends to an inner automorphism of $V_Q$ which we denote $\zeta$ again; $\sigma$ and $\zeta$ are conjugate under an inner automorphism (in the contexts of both $\fin$ and $V_Q$) \cite[Section 4]{GSV}.

Let $M_\sigma$ denote the principal realization of $\vacm$ and $M_\zeta$ the realization corresponding to $\zeta$ \cite[Section 4]{GSV}. We have that $M_\sigma$ (resp. $M_\zeta$) is a $\sigma$-twisted (resp. $\zeta$-twisted) representation of $V_Q$. Both $M_\sigma$ and $M_\zeta$ become representations of the affine Lie algebra $\aff$ (see \cite[Section 4]{GSV}) and each is isomorphic to $\vacm$.

\subsection{} Since $\walg$ is pointwise fixed by every inner automorphism of $V_Q$, the restriction of $M_\sigma$ to $\walg$ becomes an untwisted $\walg$-module. We let the action of the 
generators $\omd[p]$ on $M_\sigma$ be given by the modes $\omd[p]_i(\sigma)$ for $i \in \integers$.
The following is our main theorem.
\begin{theorem}\label{thm:mainthm}
Let $\aff$ be a simply-laced affine Lie algebra. Consider the following vectors of $M_\sigma \cong L(\Lambda_0)$:
\begin{equation}\label{eq:ffbasis-twisted}
\omd[p_1]_{k_1}(\sigma)\; \omd[p_2]_{k_2}(\sigma)\; \cdots\; \omd[p_r]_{k_r}(\sigma) \,\hwvec,
\end{equation}
\smallskip
where (i) $r \geqslant 0$\quad(ii) $\ell \geqslant p_1 \geqslant p_2 \geqslant \cdots \geqslant p_r \geqslant 1$ \quad(iii) $k_j \leqslant -1$ for all $j$\quad   and (iv) if $p_i = p_{i+1}\,$, then $k_i \leqslant k_{i+1}$.

The subspace $F^d Z_n$ of the Brylinski filtration on $Z$ has a basis given by the vectors in \eqref{eq:ffbasis-twisted} satisfying $\sum_{i=1}^r d_{p_i} \leq d$ and $\sum_{i=1}^r k_i = -n$.
\end{theorem}

\subsection{}
Given $\bar{\lambda} \in \mhb^*$, let $\gamma_{\bar{\lambda}}: U\fin \to \complex$ denote the corresponding central character of $\fin$. We let $\wverma[\bar{\lambda}]$ denote the Verma module of $\walg$ indexed by 
$\gamma_{\bar{\lambda}}$ (see Remark~\ref{rem:w-and-wk} and \S\ref{sec:wkmodule}). We refer the reader to sections~\ref{sec:wstructure}, \ref{sec:wreps} for a discussion on Verma modules of $\walg$. The following key result, of independent interest, is the essential ingredient in our proof of Theorem~\ref{thm:mainthm}:
\begin{theorem}\label{thm:zwalg}
  Let $\aff$ be a simply-laced affine algebra. The space $Z$ is a $\walg$-invariant subspace of $M_\sigma$. Further, $Z$ is an irreducible $\walg$-module and
  is isomorphic to $\wverma[\rhob/\coxnum - \rhob]$ as $\walg$-modules.
\end{theorem}
In order to prove Theorem~\ref{thm:zwalg}, it will be more convenient to work with the realization $M_\zeta$. We recall from \cite[Section 6]{GSV} that:
\begin{equation}\label{eq:msigmzeta}
M_\zeta \cong M_\sigma \text{ as } \walg\text{-modules.} 
\end{equation}
We note that $\walg \subset \fock \subset V_Q$. In the realization $M_\sigma$, $Z$ is not $\fock$-invariant. 
In $M_\zeta$, $Z$ is $\fock$-invariant and becomes an untwisted $\fock$-module. Further, $Z$ is isomorphic to the irreducible highest weight representation $\pi_{1, \,\rhob/\coxnum}$ of $\fock$ with highest weight $\rhob/\coxnum$ \cite[Section 7]{GSV}. We now have the following general theorem concerning restrictions of $\fock$-modules to $\walg$. 
\begin{theorem}\label{thm:wfock}
Consider $\walg \subset \fock$. Let $\pi_{1, \bar{\lambda}}$ be an irreducible highest weight representation of $\fock$ with highest weight $\bar{\lambda}$. If $(\bar{\lambda} | \beta) \notin \integers$ for all roots $\beta$ of $\fin$, then the restriction to $\walg$ of $\pi_{1, \bar{\lambda}}$ is irreducible, and is isomorphic to the Verma module $\wverma[\bar{\lambda}-\rhob]$.
\end{theorem}

The argument that Theorem~\ref{thm:mainthm} follows from Theorem~\ref{thm:zwalg} was formulated in a type independent manner in \cite[\S 9]{GSV} and can be applied as-is here. In view of Equation~\eqref{eq:msigmzeta} and the fact that $\bar{\lambda} = \rhob/\coxnum$ satisfies the hypothesis of Theorem~\ref{thm:wfock}, it is clear that Theorem~\ref{thm:zwalg} follows from Theorem~\ref{thm:wfock}. Thus, it remains only to prove Theorem~\ref{thm:wfock}; this is done in Sections~\ref{sec:wreps}, \ref{sec:heismodw}.

We remark here that in \cite{GSV}, we only showed (in type $A$) that $Z$ is an irreducible Verma module, but did not identify its highest weight. In this paper, we identify that precisely.



\section{Structure of $\wk$}\label{sec:wstructure}
\subsection{Universal affine vertex algebra} Let $k \in \mathbb{C}$. The universal affine vertex algebra $V_{k}(\fin)$ is defined as
$$V_{k}(\fin) := U(\mg) \otimes_{U(\fin \otimes \complex [t] \oplus \complex K \oplus \complex d)}{\complex},$$ where $\fin \otimes \complex [t] \oplus d$ acts trivially on $\complex$ and $K$ acts as multiplication by $k$.
The $\mg$-module $V_{k}(\fin)$ has a natural vertex algebra structure with ${\bf{1}} = 1 \otimes 1$ playing the role of the vacuum vector.
For $x \in \fin$, the field corresponding to the element $(x \otimes t^{-1}){\bf{1}}$ is given by
$$Y((x \otimes t^{-1}){\bf{1}}, z) = x(z) : = \sum_{n \in \mathbb{Z}}{((x \otimes t^{n}) \otimes 1) z^{-n -1}}.$$

\subsection{$\mathscr{W}$-algebras}
Let $\fin = 
\bar{\mathfrak{n}}_{-} \oplus \mhb \oplus \bar{\mathfrak{n}}_{+}$ be the 
standard triangular decomposition of $\fin$.
Let $f$ be a principal nilpotent element of $\bar{\mg}$, and $\{e, h, f\}$ be the $\mathfrak{sl}_2$-triple associated with $f$. 
Let us recall $H^{\bullet}_{DS}(\cdot)$, the Drinfeld-Sokolov ``+" reduction cohomology functor associated with $f$ (for complete definition see \cite{acl, BFEF, FKW}):
$$H^{\bullet}_{DS}(M) = H^{\frac{\infty}{2} + \bullet}(L \,\bar{\mathfrak{n}}_{-} , M \otimes \complex_{\eta}).$$
Here $M$ is a $\mg$-module from category $\mathcal{O}$ on which $K$ acts as $kI$. The action of $L \,{\bar{\mathfrak{n}}}_{-} := {\bar{\mathfrak{n}}}_{-} \otimes \complex[t, t^{-1}]$ on $M \otimes \complex_{\eta}$ is diagonal and the character $\eta: L \,{\bar{\mathfrak{n}}}_{-} \rightarrow \complex,$ $\eta(y \otimes t^n) = \delta_{n, -1}(f | y)$ gives the one-dimensional representation $\complex_{\eta}$ of $L \,{\bar{\mathfrak{n}}}_{-}$ \cite{acl}.
\begin{definition}
Let $k\in \complex$. The (principal) $\mathscr{W}$-algebra associated with $\bar{\mg}$ at level $k$ is defined as
$$\mathscr{W}_{k}(\bar{\mg}) = H^{0}_{DS}(V_{k}(\bar{\mg})).$$ 
\end{definition}
The $\mathscr{W}$ algebra $\mathscr{W}_{k}(\bar{\mg})$ has a vertex algebra structure for all $k \in \mathbb{C}$ (see Chapter 15 of \cite{FB} for details). When $k \neq -\hc$, $\wk$ admits a conformal vector $\omega$ with central charge $c(k)$ given by equation (245) of \cite{arakawa2007}:
\begin{equation}\label{eq:centralcharge}
c(k) = \rank \fin - 12 \left( (k+\hc) |\rhob^\vee|^2 - 2\langle \rhob, \rhob^\vee\rangle + \frac{|\rhob|^2}{k+\hc} \right).
\end{equation}

We now restate the Feigin-Frenkel theorem on generating fields for $\wk$ for general $k \in \complex$:
\begin{theorem} (\cite[Proposition 4.12.1]{arakawa2007}, \cite{FF})
There exist elements $\omd[i] \in \mathscr{W}_{k}(\bar{\mg})$ with degree (conformal weight) $d_i$, $1 \leq i \leq \ell$, such that  $\mathscr{W}_{k}(\bar{\mg})$ is freely generated by $\{\omd[i] : 1 \leq i \leq \ell\}.$ 
\end{theorem}

\subsection{The Miura map}
We recall the definition of the Heisenberg vertex algebra. 
Let $\mhb$ be a finite dimensional vector space equipped with a symmetric non-degenerate bilinear form.  Let $\mathcal{H} = \mhb \otimes \mathbb{C}[t, t^{-1}] \oplus \complex K$ be the Heisenberg Lie algebra.
The Fock space of $\mathcal{H}$ at level $k \in \mathbb{C}$ is
\begin{equation} \label{eq:fockdef}
  \pi_k = \mathrm{Ind}_{\mathcal{H}^{+} \oplus \complex K}^{\mathcal{H}}\,{\complex} = U(\mathcal{H}) \otimes_{U(\mathcal{H}^{+} \oplus \complex K)} \complex
  \end{equation}
where $\mathcal{H}^{+} = \mhb \otimes \complex[t]$ acts trivially on $\complex$ and $K$ acts as $k$.  As a vector space $\pi_k$ is isomorphic to the symmetric algebra $S(\mhb \otimes t^{-1}\complex[t]^{-1})$. There is a grading on $\pi_k$ given by
$\pi_k = \bigoplus_{s \geq 0}{\pi_k^{[s]}}$, where ${\pi_{k}}^{[s]}$ is spanned by $(\mhb \otimes t^{k_1})\cdots (\mhb \otimes t^{k_r})$, with the condition $k_i < 0$ and $\sum{k_i} = s$.
The Fock space $\pi_k$ has a vertex algebra structure, with the state-field correspondence
$$Y(h \otimes t^{-1}, z)= \sum_{s \in \mathbb{Z}} (h \otimes t^{s})\, z^{-s-1}$$ for all $h \in \mhb.$ 
The vertex algebra $\pi_k$ is strongly generated by $\mhb \otimes t^{-1}$.

Let $\hc$ denote the dual Coxeter number of $\fin$. Let $k \neq -\hc$. We recall that there is an injective map of vertex algebras
\begin{equation}\label{eq:miura}
  \Upsilon: \wk \rightarrow \pi_{k + \hc}
\end{equation}
called the {\em Miura map} \cite{arakawa2017, FF1}.  It maps the conformal vector of $\wk$ to a conformal vector of degree 2 in the standard grading of $\pi_{k+\hc}$:
\begin{equation}\label{eq:kw-conf}
  \Upsilon(\omega) = \omega^{\scriptscriptstyle \mathrm{Sug}} +  \left(1 - \frac{1}{k+\hc}\right)\,\rhob_{(-2)}\bf{1}
  \end{equation}
where $\omega$ is the conformal vector of $\wk$ and $\omega^{\scriptscriptstyle \mathrm{Sug}}$ is the standard Sugawara conformal vector $\pi_{k + \hc}$ \cite{KW22}. Hence $\Upsilon$ preserves the grading on both sides.

\subsection{The $k+\hc=1$ case}\label{sec:k1hc}
The image of $\Upsilon$ (equation~\eqref{eq:miura}) generically coincides with the intersection of kernels of {\em screening operators} (see for example \cite[Lemma 5.4]{acl}). It is well-known that this fact holds for $k=1-\hc$ as well \cite[Remark 4.4]{FKRW}; this can be easily verified by appropriately modifying the proof of \cite[Proposition~5.5]{acl} and using the equality of characters argument of \cite[Theorem 4.2]{FKRW}. Since the intersection of the kernels of screening operators for $k=1-\hc$ coincides with $\walg$ (equation~\eqref{eq:wdef} and \cite{acl}), we have:

\begin{prop}\label{prop:miuraiso}
  Let $k = 1-\hc$ and consider the Miura map $\Upsilon: \wk[1-\hc] \rightarrow \pi_{1}$. Then
  the image of $\Upsilon$ is $\walg$. 
\end{prop}
\begin{remark}\label{rem:w-and-wk}
Thus $\Upsilon$ defines an isomorphism between $\wk[1-\hc]$ and $\walg$. We identify these two vertex algebras via $\Upsilon$ and consider $\wk[1-\hc]$-modules as $\walg$-modules.  Note also from \eqref{eq:kw-conf} that $\Upsilon(\omega) = \omega^{\scriptscriptstyle \mathrm{Sug}}$ in this case.
\end{remark}
\section{Representations of $\wk$}\label{sec:wreps}
In this section, we assume $k+\hc \neq 0$ throughout. 
Recall that $\mathfrak{h} = \mhb \oplus \mathbb{C} K \oplus \mathbb{C}d$, where $K$ is the canonical central element and $d$ is the degree derivation of $\mathfrak{g}$. Let
$\mh^* = \mhb^* \oplus \complex \Lambda_0 \oplus \complex \delta$, where $\Lambda_0$ is the fundamental weight defined as $\Lambda_{0} |_{\mhb} = 0, \Lambda_{0}(K) = 1, \Lambda_{0}(d) =0$ and $\delta$ the imaginary root is given by
$\delta |_{\mhb} = 0 = \delta(K), \delta(d) = 1$. Any $\lambda \in \mh^*$ can be uniquely written as $\lambda = \bar{\lambda} + a \Lambda_0 + b \delta$, for some $a, b \in \complex$ and $\bar{\lambda} \in \mhb^*$. The 
element $\bar{\lambda}$ is called as the classical part of $\lambda$. The following is an useful observation:
  \begin{equation}\label{eq:lambar}
    (\lambda | \alpha) = (\bar{\lambda} | \alpha) \text{ if } \alpha \in \mhb^*
  \end{equation}

Let $\omega$ denote the conformal vector of $\wk$ (see \cite[\S 4.17]{arakawa2007} for the formula of the corresponding central charge). Given a $\wk$-module $M$, let us denote the corresponding Virasoro modes $\omega_n$, i.e.,
\begin{equation}\label{eq:omegamodes}
  Y_M(\omega,z) = \sum_{n \in \integers} \omega_n \, z^{-n-2}
\end{equation}
If $\omega_0$ acts semisimply on $M$ with finite-dimensional eigenspaces, we define the character $\ch M:= q^{\tr \omega_0}$. 

Let ${\rho} \in \mh^*$ denote the Weyl vector of $\mg$. We have $\rho =  h^{\vee} \Lambda_0 + \bar{\rho}$ where $\bar{\rho} \in \mhb^*$ denotes the classical part of $\rho$. 


\subsection{}\label{sec:wkmodule}
We now consider the Verma modules of $\wk$ \cite[\S 5.1]{arakawa2007}. Let $\mhk^*$ denote the set of $\lambda \in \mh^*$ such that $\lambda(K)=k$. Given $\lambda \in \mhk^*$, let $\bar{\lambda} \in \bar{\mh}^*$ denote its classical part. Consider the central character $\gamma_{\bar{\lambda}}: Z (U(\fin)) \rightarrow \complex$. Let $\wverma$ be the Verma module of $\mathscr{W}_{k}(\bar{\mg})$, and let ${\bf{L}}(\gamma_{\bar{\lambda}})$ be its unique irreducible quotient. Likewise, let $M(\Lambda)$ and $L(\Lambda)$ denote the Verma module of $\mathfrak{g}$ with highest weight $\Lambda \in \mathfrak{h}^*$ and its unique irreducible quotient respectively. 
We now have:
\begin{theorem}\label{thm:wverirred} \cite[Theorem 7.7.1]{arakawa2007}
  Let $k \neq -h^{\vee}$ and $\lambda \in \mhk^*$. Suppose the classical part $\bar{\lambda}$ of $\lambda$ satisfies $(\bar{\lambda} + \bar{\rho} | \bar{\alpha}^{\vee}) \notin \mathbb{N} $ for every positive root $\bar{\alpha}$ of $\fin$.
  Then
  \begin{equation}\label{eq:wml}
    \ch {\bf{L}}({\gamma_{\bar{\lambda}}}) = \sum_{\mu}{[L(\lambda): M(\mu)] \; \ch \wverma[\bar{\mu}]}.
  \end{equation}
\end{theorem}

The following theorem gives a necessary and sufficient condition for a Verma module of $\mg$ to be irreducible \cite{KK}, \cite[Theorem 2.8.1]{Etingof-book}.
\begin{theorem}\label{t1}
(Kac-Kazhdan) Let $k \in \complex$ and $\lambda \in \mhk^*$. The Verma module $M(\lambda)$ of the affine Lie algebra $\mg$ is irreducible (i.e., $M(\lambda) = L(\lambda)$) iff the following condition is satisfied:
\begin{equation}\label{eq:lamirred}
  (\lambda + {\rho} | {\alpha}) \neq \frac{N}{2} ({\alpha} | {\alpha}) \text{ for all positive roots } \alpha \text{ of } \mg \text{ and } N \in \mathbb{N}.
\end{equation}
\end{theorem}
\begin{definition}
    If $\lambda$ satisfies the condition \eqref{eq:lamirred}, we say that $\lambda$ is a {\em generic} weight. Equivalently, for $\aff$ simply-laced, $\lambda$ is generic iff the following two conditions hold:
    \begin{itemize}
        \item $(\lambda + {\rho} | {\alpha})  \notin \mathbb{N}$ for all positive real roots $\alpha$ of $\aff$, and
        \item $(\lambda + {\rho} | \delta) \neq 0$.
    \end{itemize}
\end{definition}

The following result will be important in what follows.
\begin{prop}\label{prop:wver-irr}
Let $k\neq -h^{\vee}$ and $\lambda \in \mhk^*$. If $\lambda$ is generic, then $\wverma$ is an irreducible $\wk$-module.
\end{prop}
\begin{proof}
  If $\lambda$ is generic, then Theorem~\ref{t1} implies the equality of $\mg$-modules $M(\lambda) = L(\lambda)$. Further, by choosing $\alpha$ in \eqref{eq:lamirred} to be positive roots of $\fin$, it follows from \eqref{eq:lambar} that the classical part $\bar{\lambda}$ satisfies the hypothesis of Theorem~\ref{thm:wverirred}. Thus \eqref{eq:wml} holds. But since $[L(\lambda): M(\mu)] = \delta_{\lambda\mu}$, it follows that from \eqref{eq:wml} that $\lverma = \wverma$.
\end{proof}

To end this subsection, we recall from \cite[Proposition 5.6.6]{arakawa2007} the expression for the character of the Verma module $\wverma$ when $k \neq -\hc$:
\begin{equation}\label{eq:vermachar}
  \ch \wverma  = \tr \, q^{\omega_0} = \frac{q^{\frac{|\bar{\lambda}+ \bar{\rho}|^2}{2(k+\hc)} + \frac{c(k)-\rank\fin}{24}}}{ {\varphi(q)}^{\rank \fin}},\end{equation}
where $\omega_0$ is zeroth node of conformal vector $\omega$ (see \eqref{eq:omegamodes}), $\varphi(q) = \prod_{n \geq 1} (1-q^n)$ and $c(k)$ is given by \eqref{eq:centralcharge}.


\subsection{Drinfeld-Sokolov reduction and Verma modules}
We recall that $\hds$ is a functor from the category $\mathcal{O}_k$ (comprising $\mg$-modules in category $\mathcal{O}$ on which $K$ acts as $kI$) to the category of graded $\wk$-modules \cite[\S 6.4]{arakawa2007}. We will need the following fact concerning the image of affine Verma modules under the Drinfeld-Sokolov reduction functor.
\begin{theorem}\label{thm:lrbgen}
  Let $k\neq -h^{\vee}$ and $\lambda \in \mhk^*$. If $\lrb$ is generic, then $\hds(M(\lambda))$ is an irreducible $\wk$-module. Further,  \[\hds(M(\lambda)) \cong \wverma[\lbrb].\]  
\end{theorem}
\begin{proof}
  By Propositions 9.2.1 and 9.2.3 of \cite{arakawa2007} (where $\hds$ is denoted $H^0_+$), there exists a nontrivial homomorphism of $\wk$-modules $\varphi: \wverma[\lbrb] \to \hds(M(\lambda))$. Since $\lrb$ is generic, Proposition~\ref{prop:wver-irr}  implies that $\wverma[\lbrb]$ is irreducible, and hence that the homomorphism $\varphi$ is an injection. Further the grading on both modules is induced by $\omega_0$ and $\varphi(\omega_0 v) = \omega_0 (\varphi(v))$ for all $v \in \wverma[\lbrb]$; thus $\varphi$ preserves the grading. Now, we also have that the characters of the two modules coincide by \cite[Proposition 9.2.1]{arakawa2007}. This implies that $\varphi$ must be an isomorphism.
\end{proof}

\section{Heisenberg algebra modules and $\walg$-modules}\label{sec:heismodw}
\subsection{} For $\mu \in \mhb^*$, let $\pi_{k+\hc,\,\mu}$ denote the highest weight irreducible module of the Heisenberg vertex algebra $\pi_{k+\hc}$ with highest weight $\mu$. This is defined analogously to \eqref{eq:fockdef}:
\begin{equation} \label{eq:fockmod}
  \pi_{k+\hc,\, \mu} = U(\mathcal{H}) \otimes_{U(\mathcal{H}^{+} \oplus \complex K)} \complex_\mu
  \end{equation}
where $K$ acts on $\complex_\mu$ as $k+\hc$ and $(h \otimes t^n) \in \mathcal{H}^{+}$ as $\delta_{n,0} \,\mu(h)$.

Given any module over  $\pi_{k + \hc}$, we view it as a $\wk$-module by pulling back along $\Upsilon$ \eqref{eq:miura}.
We now have the following:
\begin{theorem}\label{thm:lrbgenheis}
  Let $k \neq -\hc$ and $\lambda \in \mhk^*$. If $\lrb$ is generic, then $\pi_{k+\hc,\,\bar{\lambda}}$ (viewed as a $\wk$-module via $\Upsilon$-pullback) is irreducible and further
  \[ \pi_{k+\hc,\,\bar{\lambda}} \cong \wverma[\lbrb]. \]
 \end{theorem}
\begin{proof}
Let $\wakim$ denote the Wakimoto module with highest weight $\lambda \in \mhk^*$ \cite{acl,FF90, F05}. 
By \cite[Lemma 5.2]{acl} (cf. \cite{BFEF}), we have
\[\hds(\wakim) \cong \pi_{k+\hc,\, \lambda} \]
as $\wk$-{modules} (in fact as $\pi_{k+\hc}$-modules). The canonical map $M(\lambda) \to \wakim$ of $V_k(\fin)$-modules induces a non-trivial map of $\wk$-modules  (cf. proof of Proposition 6.2 of \cite{acl}):
  \[\psi: \hds(M(\lambda)) \to \hds(\wakim) \cong \pi_{k+\hc,\, \lambda}.\]
 Since $\lrb$ is generic, theorem~\ref{thm:lrbgen} implies that
 $\hds(M(\lambda)) \cong \wverma[\lbrb]$ is irreducible. Thus $\psi$ is an injection.

 From \eqref{eq:vermachar}, the character of $\wverma[\lbrb]$ can be computed easily:
 \[ \ch \wverma[\lbrb] = \frac{q^{h_\lambda}}{\varphi(q)^{\rank \fin}} \]
 where $h_\lambda = \frac{(\lb, \lb + 2\rhob)}{2(k+\hc)} - (\bar{\lambda},\rhob)$ (see also (30) of \cite{acl}).
 
 To compute the character of $\pi_{k + \hc, \,\lambda}$ (viewed as a $\wk$-module), we need to compute $q^{\tr \tilde{\omega}_0}$ on this module, where $\tilde{\omega} = \Upsilon(\omega)$ as in \eqref{eq:kw-conf}. The highest weight vector of $\pi_{k + \hc, \,\lambda}$ is an eigenvector of $\tilde{\omega}_0$ with eigenvalue:
 \[ \frac{(\lb, \lb)}{2(k+\hc)} - \left( 1- \frac{1}{k+\hc}\right) (\bar{\lambda},\rhob) \]
which simplifies to $h_\lambda$. Hence the characters of the two modules $\wverma[\lbrb]$ and $\pi_{k+\hc,\, \lambda}$ are equal. The injectivity of $\psi$ together with the fact that it preserves the grading implies that $\psi$ is an isomorphism.
\end{proof}

\subsection{}\label{sec:lastsec} Finally, we specialize to the case $k+\hc = 1$ that is of interest to us.
\begin{prop}\label{prop:khc1}
Let $k=1 - \hc$ and let $\lambda \in \mhk^*$ such that its classical part $\bar{\lambda}$ satisfies the property:
\[ (\bar{\lambda} | \beta)  \notin \integers \;\; \text{ for all roots } \beta \text{ of } \fin. \] 
Then $\lambda - \rhob$ is a generic weight in $\mhk^*$.
\end{prop}
\begin{proof}
  We have $\lambda = k\Lambda_0 + \bar{\lambda}$ and $\rho = \hc\Lambda_0 + \rhob$. Thus $\lambda - \rhob + \rho = (k+\hc)\Lambda_0 + \bar{\lambda} = \Lambda_0 + \bar{\lambda}$. To check that \eqref{eq:lamirred} holds, first let $\alpha$ be a positive real root of $\mg$, say $\alpha = \beta + n\delta$ for some root $\beta$ of $\fin$. Now
  \[ (\Lambda_0 + \bar{\lambda}, \alpha) = n + (\bar{\lambda}, \beta) \notin \integers.\]
  Likewise, for $\alpha = n\delta$ with $n>0$, we have  $(\Lambda_0 + \bar{\lambda}, \alpha) = n \neq 0$. Thus \eqref{eq:lamirred} holds. 
  \end{proof}

Proposition~\ref{prop:miuraiso} implies that the Verma modules $\wverma$ of $\wk[1-\hc]$ may be considered as $\walg$-modules via the identification $\Upsilon$.
It is now clear that Theorem~\ref{thm:wfock} follows readily from Theorem~\ref{thm:lrbgenheis} and Propositions~\ref{prop:miuraiso}, \ref{prop:khc1}. \qed


\bibliographystyle{plain}
\bibliography{WalgebraADE}

\end{document}